\documentclass[pdflatex,sn-mathphys-num]{sn-jnl}

\usepackage{graphicx}%
\usepackage{multirow}%
\usepackage{amsmath,amssymb,amsfonts}%
\usepackage{amsthm}%
\usepackage{mathrsfs}%
\usepackage[title]{appendix}%
\usepackage{xcolor}%
\usepackage{textcomp}%
\usepackage{manyfoot}%
\usepackage{booktabs}%
\usepackage{algorithm}%
\usepackage{algorithmicx}%
\usepackage{algpseudocode}%
\usepackage{listings}%

\theoremstyle{thmstyleone}%
\newtheorem{theorem}{Theorem}[section]

\newtheorem{lemma}[theorem]{Lemma}
\newtheorem{corollary}[theorem]{Corollary}

\theoremstyle{thmstyletwo}%

\theoremstyle{thmstylethree}%

\begin{document}

\title[The Monoid Of Binary Relations On A Set Of Size Four Has Infinite Representation Type]{The Monoid Of Binary Relations On A Set Of Size Four Has Infinite Representation Type}


\author{\fnm{Joseph Daynger} \sur{Ruiz}}\email{jdruiz@arizona.edu}

\affil{\orgdiv{Mathematics}, \orgname{University of Arizona}, \orgaddress{\street{Santa Rita}, \city{Tucson}, \postcode{85721}, \state{Arizona}, \country{USA}}}


\abstract{The problem of determining the representation type of the full transformation monoid was resolved by Ponizovskii, Putcha, and Ringel.
In this paper, we present a similar result for the monoid of binary relations, a partial result for the monoid of $\Lambda$-generalized correspondences, and a sufficient condition for a monoid to have infinite representation type.}



\maketitle

\section{Introduction}\label{sec1}

The set of all functions from a set $X$ to itself under composition is called the full transformation monoid on $X$, denoted $T_X$.
When $X$ is finite and $k$ is a field of characteristic 0, Ponizovskii \cite{Poni90}, Putcha \cite{PUTCHA98}, and Ringel \cite{Ring00} showed that the monoid algebra $k T_X$ is of finite representation type if $|X| \leq 4$ and is of infinite representation type otherwise.

The set of all relations on a set $X$ under composition of relations is called the monoid of binary relations on $X$, denoted $B_X$. 
When $X$ is finite and $k$ is a field, Pérez computed the quiver and relations of the monoid algebra $k B_X$ if $|X| \leq 3$ \cite{Perez2020}. In this case, the monoid algebra $k B_X$ is of finite representation type.
In this paper, we show that $k B_X$ is of infinite representation type otherwise.

Let $\Lambda$ be a finite distributive lattice.
A $\Lambda$-generalized correspondence on a set $X$ is a function from $X \times X$ to $\Lambda$.
Let $R$ and $S$ be two $\Lambda$-generalized correspondences on a set $X$, their product $RS$ is defined by
$$RS(x,z) = \bigvee_{y\in X}R(x,y)\wedge S(y,z)$$
for all $x,z \in X$. 
Guillaume studied these objects in the context of category theory \cite{GUILLAUME2019405}.
The set of all $\Lambda$-generalized correspondences on a set $X$ with the product defined above is a monoid which will be called the monoid of $\Lambda$-generalized correspondences on $X$, denoted $\Lambda^{X \times X}$. As a corollary of the main result, when $X$ is finite and $k$ is a field, the monoid algebra $k \Lambda^{X \times X}$ is of infinite representation type if $|X| \geq 4$ and $\Lambda$ is a nontrivial distributive lattice.

\section{Irreducibility in Monoids}\label{sec2}

A non-unit $p$ in a monoid $M$ is called \textbf{irreducible} if $p=ab$, where $a,b \in M$, implies $a$ is a unit or $b$ is a unit.
A non-unit $r$ in a monoid $M$ is called \textbf{reducible} if there exists some $a,b \in M$ which are not units and satisfy $r=ab$.
Two elements $a,b$ in a monoid $M$ are called \textbf{associates} if there exist two units $u,v \in M$ such that $a = ubv$.

\begin{lemma} \label{M to Kxy}
    Let $M$ be a finite monoid and $k$ be a field. If there exist two irreducible elements $p,q \in M$  that are not associates then the map $f:M \rightarrow k[x,y]/\langle x^2,y^2,xy \rangle$
    defined as follows
    $$f(m) = \begin{cases}
    1 & \text{if } m \text{ is a unit} \\
    x & \text{if } m \text{ and } p \text{ are associates}\\
    y & \text{if } m \text{ and } q \text{ are associates}\\
    0 & \text{otherwise}
    \end{cases}$$
    is a well defined homomorphism of monoids and induces a surjective $k$-algebra homomorphism between $kM$ and $k[x,y]/\langle x^2,y^2,xy \rangle$.
\end{lemma}
\begin{proof}
    The map $f$ is well defined since every element of $M$ satisfies exactly one condition given by the piecewise function and so each element is only mapped to one output.
    We will now show that $f$ is a homomorphism.
    Let $a,b \in M$.
    We will first assume that at least one of the elements is a unit.
    Without loss of generality, we assume that $a$ is a unit.
    Note that $ab$, $ba$, and $b$ are pairwise associate, so under the map $f$ they have the same output.
    That is, $f(ab)=f(ba)=f(b)$.
    Since $a$ is a unit and $f$ maps units to the identity, we can write $f(ab)=1 \cdot f(b)$ and $f(ba)=f(b) \cdot 1$ as $f(ab)=f(a)f(b)$ and $f(ba)=f(b)f(a)$, respectively.
    This shows that $f$ satisfies the homomorphism property when at least one element is a unit.
    Now consider the case that both elements, $a$ and $b$, are not units.
    The product $ab$ is reducible by the definition of reducibility. 
    Since $ab$ is reducible, it cannot be a unit or associate to an irreducible element.
    Thus the image of $ab$ under the map $f$ is zero.
    The image of $a$ and $b$ lie in the ideal generated by $x$ and $y$.
    The square of this ideal is zero, so the product of the images of $a$ and $b$ is also zero.
    That is, $f(ab)=f(a)f(b)$.
    This gives us that $f$ is a homomorphism of monoids.

    The monoid algebra $k M$ satisfies the universal property that for every monoid homomorphism $g$ from $M$ to a $k$-algebra $A$ there exists a unique $k$-algebra homomorphism $\Tilde{g}$ from $k M$ to $A$ which satisfies $g(m)=\Tilde{g}(1m)$ for all $m \in M$. 
    This allows us to lift $f$ to a $k$-algebra homomorphism $\Tilde{f}$ from $kM$ to $k[x,y]/\langle x^2,y^2,xy \rangle$ which has the basis $\{1,x,y\}$ in its image.
    This gives us a surjective $k$-algebra homomorphism between $kM$ and $k[x,y]/\langle x^2,y^2,xy \rangle$, which completes the proof.
\end{proof}

A consequence of Lemma \ref{M to Kxy} will be that for any finite monoid $M$ satisfying the conditions in Lemma \ref{M to Kxy} and any field $k$, the monoid algebra $kM$ is of infinite representation type. We recall the following lemma, which is essential for our purpose.

\begin{lemma} \label{k3}
    For any field $k$, the 3-dimensional commutative $k$-algebra $$k[x,y]/\langle x^2,y^2,xy \rangle$$ is of infinite representation type.
\end{lemma}
\begin{proof}
    See Lemma 8.5 in \cite{Erdmann2018}.
\end{proof}

We also need the following lemma.

\begin{lemma} \label{AmodI}
    Let $k$ be a field, $A$ be a $k$-algebra, and $I \subseteq A$ be a two-sided ideal with $I \neq A$. If the factor algebra $A/I$ is of infinite representation type, then $A$ is of infinite representation type.
\end{lemma}
\begin{proof}
    See Lemma 8.6 in \cite{Erdmann2018}.
\end{proof}

The following corollary is a consequence of Lemma \ref{M to Kxy}, \ref{k3}, and \ref{AmodI} and is nearly sufficient to prove the claim in the title of this paper.

\begin{corollary} \label{elms to rep}
    Let $M$ be a finite monoid and $k$ be a field. If there exist two irreducible elements $p,q \in M$ that are not associates, then $kM$ is of infinite representation type.
\end{corollary}

Predictably, the key step now is to demonstrate that the monoid of binary relations on a set of four elements has two irreducible elements which are not associates. 

\begin{lemma} \label{its clear when you know them}
    The monoid of relations on a set of four elements has two irreducible elements that are not associates.
\end{lemma}
\begin{proof} See Example 2.5 in \cite{DECAEN1981119} with the knowledge that the boolean matrix algebra of all $n$ by $n$ matrices and the monoid of relations on a set of size $n$ are isomorphic.
\end{proof}

Lemma \ref{its clear when you know them} and Corollary \ref{elms to rep} prove the following theorem, which is the theorem focused on in the title of the paper.

\begin{theorem} \label{Big 1}
    The monoid of relations on a set of four elements has infinite representation type over any field.
\end{theorem}

We now turn our attention to the monoid of relations on sets of size five or greater and the monoid of $\Lambda$-generalized correspondences on sets of size four or greater that were introduced at the beginning of the paper.
When $\Lambda$ is the two-element distributive lattice and $X$ is a set, the monoid of relations on $X$, $B_X$, and the monoid of $\Lambda$-generalized correspondences on $X$, $\Lambda^{X \times X}$, are isomorphic.
For a finite set $X$, proving that the monoid algebra $k \Lambda^{X \times X}$ has infinite representation type for any field $k$ and any nontrivial distributive lattice $\Lambda$ also proves that the monoid algebra $k B_X$ has infinite representation type due to the aforementioned isomorphism.
Because of this, we will focus the remainder of the discussion on the monoid of $\Lambda$-generalized correspondences.

\section{Induction in Monoid Algebras}\label{sec3}

Let $k$ be a field and $A$ be a finite dimensional $k$-algebra.
An element $e$ in $A$ is called an idempotent if that element satisfies $e^2 = e$.
Given an idempotent $e$ in $A$, the subspace $eAe$ is a $k$-algebra with identity $e$ where the product is the product on $A$ restricted to $eAe$.
Furthermore, $Ae$ can be given the structure of an $A-eAe$ bimodule.
Using an idempotent $e$ in $A$, we can construct a functor from the category of left $eAe$-modules to the category of left $A$-modules denoted $\operatorname{Ind}_e$ and defined as follows for a left $eAe$-module $V$:
$$\operatorname{Ind}_e(V) = Ae \otimes_{eAe} V\text{.}$$
This functor is called induction.

\begin{lemma} \label{ind is good}
    The functor $\operatorname{Ind}_e$ is fully faithful, and if $V$ is an indecomposable left $eAe$-module then $\operatorname{Ind}_e(V)$ is an indecomposable left $A$-module.
\end{lemma}
\begin{proof}
    See Proposition 4.6 and Corollary 4.10 in \cite{St16}.
\end{proof}

The following is an immediate consequence of Lemma \ref{ind is good}. 

\begin{corollary} \label{almost there}
    Let $k$ be a field and $A$ be a finite dimensional $k$-algebra. If $e$ in $A$ is an idempotent different from the identity of $A$ such that $eAe$ is of infinite representation type then the $k$-algebra $A$ is of infinite representation type.
\end{corollary}

Corollary \ref{almost there} becomes more monoid theoretic when we consider idempotents in a monoid.
Furthermore, the following corollary, along with Theorem \ref{Big 1}, almost proves the remaining claims in the abstract.

\begin{corollary} \label{even closer}
    Let $M$ be a finite monoid and $k$ be a field. If there exists an idempotent $e$ in $M$ such that the monoid $eMe$ has two irreducible elements that are not associates, then $kM$ is of infinite representation type.
\end{corollary}

We now conclude by showing that the monoid of $\Lambda$-generalized correspondences on a set $X$ of four or more elements contains an idempotent $e$ such that $e \Lambda^{X \times X} e$ is isomorphic to the monoid of relations on a set of four elements.

\begin{lemma} \label{M in L}
    Let $X$ be a set of size four or greater and $\Lambda$ be a nontrivial finite distributive lattice. There exists an idempotent $e$ in $\Lambda^{X \times X}$ such that $e \Lambda^{X \times X} e$ is isomorphic to the monoid of relations on a set of four elements.
\end{lemma}
\begin{proof}
    Choose $Y$ to be a subset of $X$ with exactly four elements.
    A partial order can be put on the distributive lattice $\Lambda$ as follows: $a \in \Lambda$ is less than or equal to $b \in \Lambda$ when $a \wedge b = a$.
    The meet of all the elements in the lattice, $\bigwedge_{a\in \Lambda}a$, is the unique minimal element of $\Lambda$ when considered as a poset with respect to this partial order.
    Let $\mathbf{0}$ denote the meet of all the elements in the lattice $\Lambda$.
    Because $\Lambda$ is a nontrivial finite distributive lattice, the minimal element $\mathbf{0}$ has a cover.
    Let $c$ denote one such cover.
    The cover $c$ of $\mathbf{0}$ has the property that the meet of $c$ and any other element in the lattice is $\mathbf{0}$ or $c$ and the join of $\mathbf{0}$ and $c$ is $c$.
    Let $e_{Y,c}$ be the element in $\Lambda^{X \times X}$ defined as follows:
    $$e_{Y,c}(x,y) = \begin{cases}
    c & \text{if } x=y \text{ and } x\in Y \\
    \mathbf{0} & \text{otherwise}
    \end{cases}.$$
    We will first show the element $e_{Y,c}$ is idempotent.
    Note that $e_{Y,c}e_{Y,c}(x,y)$ when $x$ is not in $Y$ is the join over all the meet of $\mathbf{0}$ with some other element of $\Lambda$.
    The meet of any element of $\Lambda$ and $\mathbf{0}$ is $\mathbf{0}$ and the join of any element with itself is itself, so $e_{Y,c}e_{Y,c}(x,y)$ is $\mathbf{0}$ when $x$ is not in $Y$.
    A similar argument can be made when $y$ is not in $Y$.
    When both $x$ and $y$ are in $Y$ then $e_{Y,c}e_{Y,c}(x,y)$ is the join over all the meet of $e_{Y,c}(x,z)$ with $e_{Y,c}(z,y)$ for $z$ in $X$. When $x$ does not equal $y$ at least one member of the meet $e_{Y,c}(x,z)$ with $e_{Y,c}(z,y)$ is $\mathbf{0}$ and so the join over all such meets is also $\mathbf{0}$.
    Finally when both $x$ and $y$ are in $Y$ and equal then the meet of $e_{Y,c}(x,z)$ with $e_{Y,c}(z,y)$ for $z=x=y$ is $c$ since $e_{Y,c}(a,a)$ is $c$ for $a$ in $Y$.
    Since $c$ is the cover of $\mathbf{0}$, the join of any number of $\mathbf{0}$ and $c$ is $c$.
    These cases show $e_{Y,c}e_{Y,c}$ is an is the same function from $X \times X$ to $\Lambda$ as $e_{Y,c}$ and so $e_{Y,c}e_{Y,c} = e_{Y,c}$ which means $e_{Y,c}$ is an idempotent.
    Note that an arbitrary element $\alpha$ in the set $e_{Y,c} \Lambda^{X \times X} e_{Y,c}$ satisfies $\alpha(x,y) = \mathbf{0}$ if $x \in X\setminus Y$ or $y \in X\setminus Y$ and $\alpha(x,y)$ is $\mathbf{0}$ or $c$ for all $x,y \in Y$.
    The set of elements $e_{Y,c} \Lambda^{X \times X} e_{Y,c}$ with the product given by the product on $\Lambda^{X \times X}$ restricted to $e_{Y,c} \Lambda^{X \times X} e_{Y,c}$ is a monoid and is isomorphic to $\{\mathbf{0},c\}^{Y \times Y}$ where $\{\mathbf{0},c\}$ is the two-element distributive lattice.
    Since $Y$ is a set of size four, we have that $\{\mathbf{0},c\}^{Y \times Y}$ is isomorphic to $B_Y$ which is the monoid of relations on a set of four elements.
    Thus we have shown that there is an idempotent $e \in \Lambda^{X \times X}$ such that $e \Lambda^{X \times X} e$ is isomorphic to the monoid of relations on a set of four elements.
\end{proof}

Lemma \ref{M in L}, Corollary \ref{even closer} and Lemma \ref{its clear when you know them} together prove the final theorem of the paper. 

\begin{theorem}
    The monoid of $\Lambda$-generalized correspondences on sets of size four or greater when $\Lambda$ is a nontrivial finite distributive lattice has infinite representation type over any field.
\end{theorem}

Although useful due to its monoid theoretic criterion, Corollary \ref{even closer} is not suited for regular monoids because regular monoids do not contain any irreducible elements.

\bibliography{sn-bibliography}

\end{document}